\documentclass[12pt]{article}
\usepackage{amsmath,amssymb,amsbsy,amsfonts,amsthm,latexsym,
                amsopn,amstext,amsxtra,euscript,amscd,hyperref,url}
                
\usepackage[utf8]{inputenc} 

\usepackage{geometry} 
\usepackage{graphicx} 

\usepackage{booktabs} 
\usepackage{array} 
\usepackage{paralist} 
\usepackage{verbatim} 
\usepackage{subfig} 
\usepackage{amsfonts}
\usepackage{amssymb}
\usepackage{amsmath}
\usepackage{eufrak}

\usepackage{fancyhdr} 
\newtheorem{thm}{Theorem}

\newtheorem{lemma}[thm]{Lemma}

\newtheorem{theorem}[thm]{Theorem}

\def\\{\cr}
\def\({\left(}
\def\){\right)}
\def\[{\left[}
\def\]{\right]}
\def\<{\langle}
\def\>{\rangle}

\begin{document}
\date{ }
\author{
{\sc   Bryce Kerr} \\
{Department of Computing, Macquarie University} \\
{Sydney, NSW 2109, Australia} \\
{\tt  bryce.kerr@mq.edu.au}}
\title{Rational exponential sums over the divisor function}

 \maketitle
\begin{abstract}
We consider a problem posed by Shparlinski, of giving nontrivial bounds for rational exponential sums over the arithmetic function $\tau(n)$, counting the number of divisors of $n$. This is done using some ideas of Sathe concerning the distribution in residue classes of the function $\omega(n)$, counting the number of prime factors of $n$, to bring the problem into a form where, for general modulus, we may apply a bound of Bourgain concerning exponential sums over subgroups of finite abelian groups and for prime modulus some results of Korobov and Shkredov.
\end{abstract}

\section{Introduction}
We consider a problem posed by Shparlinski~\cite[Problem~3.27]{Sp} of bounding rational exponential sums over the divisor function. More specifically, for integers $a,m$ with $(a,m)=1$ and $m$ odd  we consider the sums 
\begin{equation}
\label{main sums1}
T_{a,m}(N)=\sum_{n=1}^{N}e_m({a\tau(n)}),
\end{equation}
where $e_m(z)=e^{2\pi i z/m}$ and $\tau(n)=\sum_{d|n}1$ counts the number of divisors of $n$. Arithmetic properties of the divisor function have been considered in a number of works, see for example ~\cite{DeIw,ErdMir, Hb,LuSh}, although we are concerned mainly with  
 congruence properties of the divisor function, which have also been considered in~\cite{Co,Nz,Se}. Exponential sums over some other arithmetic functions have been considered in~\cite{BHS,BS}. Our first step in bounding the sums~\eqref{main sums1} is to give a sharper version of a result of Sathe~\cite[Lemma~1]{Se} concerning the distribution of the function $\omega(n)$ in residue classes, where $\omega(n)$ counts the number of distinct prime factors of $n$. This allows us to reduce the problem of bounding~\eqref{main sums1} to bounding sums of the form
\begin{equation}
\label{S}
S_m(r)=\sum_{n=1}^{t}e_m(r2^n),
\end{equation}
where $t$ denotes the order of $2 \pmod m$ and we may not necessarily have $(r,m)=1$.
For arbitrary $m$ we deal with these sums using a bound of Bourgain~\cite{Bou} and when $m$ is prime we obtain sharper bounds using results of Korobov~\cite{Kor} when the order of $2 \pmod m$ is not to small. For smaller values we use results of Shkredov~\cite{Sk}, which are based on previous results of Heath-Brown and Konyagin~\cite{HbKg}.

\section{Notation}
 We use the notation $f(x) \ll g(x)$ and $f(x)=O(g(x))$ to mean there exists some absolutle constant $C$ such that $f(x)\le Cg(x)$ and we use $f(x)=o(g(x))$ to mean that $f(x)\le \varepsilon g(x)$ for any $\varepsilon>0$ and sufficiently large $x$. \newline \indent
If $p|n$ and $\theta$ is the largest power of $p$ dividing $n$, we write $p^{\theta}||n$.
We let $\mathcal{S}$ denote the set of all square-free integers,
$\mathcal{M}_m$ the set of integers which are perfect $m$-th powers,  $\mathcal{Q}_m$  the set of integers $n$,  such that if $p^{\theta}||n$ then $2\le \theta \le m-1$ and   $\mathcal{K}$ the set of integers $n$ such that if $p^{\theta}||n$ then $\theta\ge 2$. \newline \indent Given an arbitrary set of integers $\mathcal{A}$, we let $\mathcal{A}(x)$ count the number of integers in $\mathcal{A}$ less than $x$, so that 
\begin{equation*}
\mathcal{Q}_m(x)\le \mathcal{K}(x) \ll x^{1/2},
\end{equation*}
hence the sums
\begin{equation}
\label{H def}
H(r,m)=\displaystyle\sum_{\substack{q \in \mathcal{Q}_m \\ \tau(q) \equiv r \pmod m}}\frac{h(q)}{q}, \quad \quad h(q)=\prod_{p|q}\left(1+\frac{1}{p}\right)^{-1},
\end{equation}
converge. We let $\zeta(s)$ denote the Riemann-zeta function, 
$$\zeta(s)=\sum_{n=1}^{\infty}\frac{1}{n^s}, \quad \Re(s)>1,$$
and $\Gamma(s)$ the Gamma function, 
$$\Gamma(s)=\int_{0}^{\infty}x^{s-1}e^{-x}dx, \quad \Re(s)>0.$$
For odd integer $m$ we let $t$ denote the order of $2 \pmod m$ and define
\begin{equation}
\label{alpha def}
\alpha_t=1-\cos(2\pi/t).
\end{equation}

\section{Main results}

\begin{theorem}
\label{thm:t1}
Suppose $m$ is odd and sufficiently large. Then with notation as in~\eqref{main sums1},~\eqref{S},~\eqref{H def} and~\eqref{alpha def} we have
$$T_{a,m}(N)=\frac{\zeta{(m)}}{t}\frac{6}{\pi^2}\left(\displaystyle\sum_{r=0}^{m-1}H(r,m)S_{m}(ar)\right)N+O\left(tN(\log{N})^{-\alpha_t}\right).$$
\end{theorem}
When $m=p$ is prime we use a different approach to save an extra power of $\log{N}$ in the asymptotic formula above, although our bound is worse in the $t$ aspect.
\begin{theorem}
\label{thm:t2}
Suppose $p>2$ is prime, then
\begin{align*}
T_{a,m}(N)&=\frac{\zeta{(p)}}{t}\frac{6}{\pi^2}\left(\displaystyle\sum_{r=0}^{p-1}H(r,p)S_{p}(ar)\right)N+O\left(pN(\log{N})^{-(\alpha_t+1)}\right).
\end{align*}
\end{theorem} 
Combining Theorem~\ref{thm:t1} with the main result from~\cite{Bou} we obtain a bound which is nontrivial for $N \ge e^{ct^{1/\alpha_t}}$ for some fixed constant $c$. 
\begin{theorem}
\label{ub m}
Suppose $m$ is odd and sufficiently large, then for all $\varepsilon>0$ there exists $\delta>0$ such that if   $t>m^{\varepsilon}$ then
$$ \max_{(a,m)=1}|T_{a,m}(N)| \ll \left(\frac{1}{m^{\delta}}+t(\log{N})^{-\alpha_t}\right)N.$$
\end{theorem}

Combining Theorem~\ref{thm:t2} with results from~\cite{Kor} and~\cite{Sk} we get,

\begin{theorem}
\label{ub p}
Suppose $p>2$ is prime and  let 
$$A(t)=\begin{cases}  p^{1/8}t^{-7/18}(\log{p})^{7/6}, \quad  \ \  \  \ t \le p^{1/2}, \\  
 p^{1/4}t^{-23/36}(\log{p})^{7/6}, \quad \  \ p^{1/2}<t \le p^{3/5}(\log{p})^{-6/5}, \\ 
 p^{1/6}t^{-1/2}(\log{p})^{4/3}, \quad \ \ \ \ \  p^{3/5}<t \le p^{2/3}(\log{p})^{-2/3}, \\
p^{1/2}t^{-1}\log{p}, \quad \quad \quad  \ \ \ \ \ \ \ t>p^{2/3}(\log{p})^{-2/3}, \end{cases}$$
then we have
$$\max_{(a,p)=1} |T_{a,p}(N)|\ll \left(A(t)+p(\log{N})^{-(\alpha_t+1)}\right)N.$$
\end{theorem}

\section{Preliminary results}

We use the decomposition of integers as in~\cite{Se}.
\begin{lemma}
\label{decom}
For integer $m$, any $n \in \mathbb{N}$ may be written uniquely in the form
$$n=sqk$$
with $s\in \mathcal{S}$, $q\in \mathcal{Q}_m$, $k \in \mathcal{M}_m$ and \ $\gcd(q,s)=1$. For such a representation, we have
$$\tau(n)\equiv \tau(s)\tau(q) \pmod m.$$
\end{lemma}
\begin{proof}
We first fix an integer $m$. Given any integer $n$, let $n=p_1^{\alpha_1}\dots p_{j}^{\alpha_j}$ be the prime factorisation of $n$, so that
\begin{equation}
\label{tau m}
\tau(n)=(\alpha_1+1)\dots (\alpha_j+1).
\end{equation}
Let $\beta_i$ be the remainder when $\alpha_i$ is divided by $m$. Then for some $k\in \mathcal{M}_m$ we have $$n=k p_1^{\beta_1}\dots b_j^{\beta_j}=k\displaystyle\prod_{\beta_i=1}p_i^{\beta_i}\displaystyle\prod_{\beta_{i}\neq 1}p_i^{\beta_i}=ksq,$$
with $s\in \mathcal{S}$, $q\in \mathcal{Q}_m$ and $\gcd(q,s)=1$. Finally, we have from~\eqref{tau m}
$$\tau(n)\equiv(\beta_1+1)\dots(\beta_j+1)\equiv \tau(qs) \equiv \tau(q)\tau(s) \pmod m,$$
since $\gcd(q,s)=1$. 
\end{proof}
Given integer $k$, we let $\omega(k)$ denote the number of distinct prime factors of $k$. The proof of the following Lemma is well known (see~\cite{Gr} and references therein for sharper results and generalizations). We provide a standard proof.
\begin{lemma}
\label{finite primes}
Suppose $q$ is squarefree and let $A_q(X)$ count the number of integers $n\le X$ such that any prime  dividing $n$ also divides $q$, then
$$A_q(X)\ll \frac{1}{\omega(q)!}\left(\log{X}+2\omega(q)^{1/2}\log{q}\right)^{\omega(q)}.$$
\end{lemma}
\begin{proof}
Suppose $p_1,\dots p_N$ are the distinct primes dividing $q$. Let $\langle .\  , .  \rangle$ denote the standard inner product on $\mathbb{R}^N$, $||.||$ the Euclidian norm and let $\mathbb{R}_{+}^N\subset \mathbb{R}^N$ be the set of all points with nonnegative coordinates. Let $P=(\log{p_1},\dots,\log{p_N})$ and 
$$\mathcal{P}(Y)=\{ \mathbf{x}\in \mathbb{R}_{+}^N : \langle \mathbf{x},P\rangle \le Y \}, $$
so that
\begin{equation}
\label{A intersect}
A_q(X)= \#(\mathbb{Z}^N \cap \mathcal{P}(\log{X})).
\end{equation}
Let  $\mathcal{C}$ denote the set of cubes of the form 
$$\left[j_1, j_1+1\right] \times \dots \times \left[j_N, j_N+1\right], \quad j_1,\dots,j_N \in \mathbb{Z},$$
which intersect $\mathcal{P}(\log{X})$, so that by~\eqref{A intersect} we have
\begin{equation}
\label{Aq}
A_q(X)\le \# \mathcal{C}.
\end{equation}
Suppose $\mathcal{B}\in \mathcal{C}$, then  for some $\mathbf{a}\in \mathbb{R}^N$ independent of $\mathcal{B}$ we have 
\begin{equation}
\label{subset inc}
\mathcal{B}\subset \mathcal{P}(\log{X}+2N^{1/2}||P||)+\mathbf{a}.
\end{equation}
 Since choosing $\mathbf{x}_0\in \mathcal{B} \cap \mathcal{P}(\log{X})$  we may write any $\mathbf{x}\in \mathcal{B}$ as $\mathbf{x}=\mathbf{x}_0+\mathbf{x}'$ with $||\mathbf{x}'||\le N^{1/2}.$ Hence by the Cauchy-Schwarz inequality and the assumption $\mathbf{x}_0 \in \mathcal{B}$ we have
$$\langle \mathbf{x},P\rangle = \langle \mathbf{x}_0,P\rangle+\langle \mathbf{x}',P\rangle \le \log{X}+N^{1/2}||P||,$$
$$\langle \mathbf{x},P\rangle \ge -|\langle \mathbf{x'},P\rangle|\ge  -N^{1/2}||P||,$$
  so that~\eqref{subset inc} holds with $\mathbf{a}=-(N^{1/2}||P||,\dots,N^{1/2}||P||)$.  Hence from~\eqref{Aq}
\begin{align*}
A_q(X)\le \# \mathcal{C}\le \displaystyle\int_{\substack{\langle \mathbf{x},P \rangle \le \log{X}+2N^{1/2}||P|| \\ \mathbf{x}\in \mathbb{R}_+^N}}1 \  d\mathbf{x}= \frac{(\log{X}+2N^{1/2}||P||)^{N}}{N!\log{p_1}\dots \log{p_N}},
\end{align*}
and the result follows since  $N=\omega(q)$ and $||P||\le \log{p_1}+\dots+\log{p_N}=\log{q}$ since $q$ is squarefree.
\end{proof}
 We use the following result of Selberg \cite{Se1}, for related and more precise results see ~\cite[II.6]{Te}.
\begin{lemma}
\label{selberg}
For any $z \in \mathbb{C}$,
$$\displaystyle\sum_{\substack{n\leq x \\ n\in \mathcal{S}}}z^{\omega(n)}=G(z)x(\log{x})^{z-1}+O\left(x(\log{x})^{\Re({z})-2}\right),$$
with $$G(z)=\frac{1}{\Gamma(z)}\displaystyle\prod_{p}\left(1+\frac{z}{p}\right)\left(1-\frac{1}{p}\right)^{z},$$
and the implied constant is uniform for all $|z|=1.$
\end{lemma}
We combine Lemma~\ref{finite primes} and  Lemma~\ref{selberg} to get a sharper version of ~\cite[Lemma~1]{Se}.
\begin{lemma}
\label{Mq}
For integers $q,r,t$  let
$$M(x,q,r,t)=\#\{ \ n\leq x \ : \ n\in \mathcal{S}, \quad \omega(n) \equiv r\pmod t, \quad  (n,q)=1 \}.$$
Then for $x\ge q$ we have
$$M(x,q,r,t)=\frac{6h(q)}{\pi^2t}x+O\left(x^{1/2}(e^4\log{x})^{\omega(q)}\right)+O\left( x (\log{x})^{-\alpha_t}\log \log{q}\right).$$ 
\end{lemma}
\begin{proof}
Suppose first $q$ is squarefree and let 
$$S(a,x)=\displaystyle\sum_{\substack{n\leq x \\ n\in \mathcal{S}}}e_t{(a\omega(n))},$$
and
$$S_1(a,q,x)=\displaystyle\sum_{\substack{n\leq x \\ n\in \mathcal{S} \\ (n,q)=1}}e_t{(a\omega(n))}.$$
Since the numbers $e_m(a\omega(n))$ with $(n,q)=1$ and $n\in \mathcal{S}$ are the coefficients of the Dirichlet series
$$\displaystyle\prod_{p  \nmid q}\left(1+\frac{e_t(a)}{p^{s}}\right)=\displaystyle\prod_{p |q}\frac{1}{\left(1+\frac{e_t(a)}{p^s}\right)}\displaystyle\prod_{p}\left(1+\frac{e_t(a)}{p^{s}}\right),$$
we let the numbers $a_n$ and $b_n$ be defined by 
\begin{align*}
\displaystyle\prod_{p |q}\frac{1}{\left(1+\frac{e_t(a)}{p^s}\right)}&=\displaystyle\sum_{n=1}^{\infty}\frac{a_n}{n^s}, \\
\displaystyle\prod_{p}\left(1+\frac{e_t(a)}{p^{s}}\right)&=\displaystyle\sum_{n=1}^{\infty}\frac{b_n}{n^s},
\end{align*}
so that $$S(a,x)=\sum_{n\le x}b_n, $$
 and
\begin{align*}
S_1(a,q,x)=\displaystyle\sum_{n\leq x}\displaystyle\sum_{d_1d_2=n}b_{d_1}a_{d_2}=\displaystyle\sum_{n\le x}a_n S(a,x/n). 
\end{align*}

Consider when $a\neq 0$,  by Lemma~\ref{selberg}
\begin{align*}
\displaystyle\sum_{n\le x}a_nS(a,x/n)&=G(e_t(a))x\displaystyle\sum_{n\le x}\frac{a_n}{n}(\log{(x/n)})^{e_t(a)-1}\\ & \quad +O\left(\displaystyle\sum_{n\le x}|a_n|\frac{x}{n}(\log{x/n})^{\cos{(2\pi /t)}-2}\right) \\
&\ll x (\log{x})^{-(1-\cos(2\pi  /t))}\sum_{n\le x}\frac{|a_n|}{n} \\
&\ll x (\log{x})^{-\alpha_t}\prod_{p|q}\left(1-\frac{1}{p}\right)^{-1},
\end{align*}
and since $$\displaystyle\prod_{p|q}\left(1-\frac{1}{p}\right)^{-1}=\frac{q}{\phi{(q)}}\ll \log \log {q},$$
where $\phi$ is Euler's totient function, we get
\begin{equation}
\label{sa}
S_1(a,q,x) \ll  x (\log{x})^{-\alpha_t}\log \log {q}.
\end{equation}
For $a=0$, by~\cite[Theorem~334]{HardyWright}
\begin{align*}
S_1(0,q,x)&=\displaystyle\sum_{n\le x}a_n S(0,x/n) \\
&=\frac{6x}{\pi^2}\sum_{n\le x}\frac{a_n}{n}+O\left(x^{1/2}\sum_{n\le x}\frac{|a_n|}{n^{1/2}} \right) \\
&=\frac{6x}{\pi^2}\displaystyle\prod_{p|q}\left(1+\frac{1}{p}\right)^{-1}+O \left(x\sum_{n\ge x}\frac{|a_n|}{n}\right)
+O\left(x^{1/2}\sum_{n\le x}\frac{|a_n|}{n^{1/2}} \right) 
.
\end{align*}
For the first error term, with notation as in Lemma~\ref{finite primes}, we have $$\sum_{n\le t}|a_n|=A_q(t),$$ so that
\begin{align}
\label{step 1}
\sum_{n\ge x}\frac{|a_n|}{n}&\ll \displaystyle\int_{x}^{\infty}\frac{A_q(t)}{t^2}dt \nonumber \\
&\ll \frac{1}{\omega(q)!}\displaystyle\int_{x}^{\infty}\frac{\left(\log{t}+2\omega(q)^{1/2}\log{q}\right)^{\omega(q)}}{t^2}dt,
\end{align}
and
\begin{align}
\label{binom expand}
\displaystyle\int_{x}^{\infty}\frac{\left(\log{t}+2\omega(q)^{1/2}\log{q}\right)^{\omega(q)}}{t^2}dt &=
\displaystyle\sum_{n=0}^{\omega(q)}\binom{\omega(q)}{n}\left(2\omega(q)^{1/2}\log{q}\right)^{\omega(q)-n}
\displaystyle\int_{x}^{\infty}\frac{\left(\log{t}\right)^{n}}{t^2}dt. 
\end{align}
The integral  $$\displaystyle\int_{x}^{\infty}\frac{\left(\log{t}\right)^{n}}{t^2}dt,$$ is the $n$-th derivative of the function
$$H(z)=\displaystyle\int_{x}^{\infty}t^{z-2}dz=\frac{x^{z-1}}{1-z},$$
evaluated at $z=0$. Hence by Cauchy's Theorem, letting $\gamma \subset \mathbb{C}$ be the circle centered at 0 with radius $1/\log{x}$ we have
\begin{align*}
\displaystyle\int_{x}^{\infty}\frac{\left(\log{t}\right)^{n}}{t^2}dt=\frac{n!}{2\pi i}\displaystyle\int_{\gamma}\frac{x^{z-1}}{1-z}\frac{1}{z^{n+1}}dz\ll \frac{n!(\log{x})^n}{x}.
\end{align*}
Hence by~\eqref{step 1} and~\eqref{binom expand}
\begin{align*}
\sum_{n\ge x}\frac{|a_n|}{n}\ll \frac{1}{x}\frac{\left(2\omega(q)^{1/2}\log{q}\right)^{\omega{(q)}}}{\omega(q)!}\sum_{n=0}^{\omega(q)}\binom{\omega(q)}{n}n!\left(\frac{\log{x}}{\omega(q)^{1/2}\log{q}}\right)^n,
\end{align*}
and by Stirling's formula~\cite[Equation~B.26]{MgVu}
\begin{align*}
\sum_{n=0}^{\omega(q)}\binom{\omega(q)}{n}n!\left(\frac{\log{x}}{\omega(q)^{1/2}\log{q}}\right)^n&\ll
\sum_{n=0}^{\omega(q)}\binom{\omega(q)}{n}n^{1/2}\left(\frac{n}{e}\right)^n\left(\frac{\log{x}}{\omega(q)^{1/2}\log{q}}\right)^n \\
&\le \omega(q)^{1/2}\sum_{n=0}^{\omega(q)}\binom{\omega(q)}{n}\left(\frac{\omega{(q)}^{1/2}\log{x}}{e\log{q}}\right)^n \\
&\ll  \omega(q)^{1/2}\left(\frac{\omega(q)^{1/2}\log{x}}{e\log{q}}+1\right)^{\omega{(q)}},
\end{align*}
so that
\begin{align*}
\sum_{n\ge x}\frac{|a_n|}{n}&\ll\frac{1}{x}\frac{\omega(q)^{1/2}\left(2\omega(q)^{1/2}\log{q}\right)^{\omega{(q)}}}{\omega(q)!} \left(\frac{\omega(q)^{1/2}\log{x}}{e\log{q}}+1\right)^{\omega{(q)}}.
\end{align*}
By another application of Stirling's formula,
\begin{align*}
\sum_{n\ge x}\frac{|a_n|}{n}&\ll \frac{2^{\omega(q)}}{x}\left(\frac{e\log{q}}{\omega{(q)}^{1/2}}\right)^{\omega(q)}\left(\frac{\omega(q)^{1/2}\log{x}}{e\log{q}}+1\right)^{\omega{(q)}} \\
&\ll \frac{2^{\omega(q)}}{x}\left(\log{x}+\frac{e\log{q}}{\omega{(q)}^{1/2}}\right)^{\omega(q)} \\
&\ll2^{\omega(q)}\left(1+\frac{3}{\omega(q)^{1/2}}\right)^{\omega(q)}\frac{(\log{x})^{\omega(q)}}{x} \\
&\ll 2^{\omega(q)}e^{3\omega(q)^{1/2}}\frac{(\log{x})^{\omega(q)}}{x}\ll e^{4\omega(q)}\frac{(\log{x})^{\omega(q)}}{x},
\end{align*}
which gives
\begin{align*}
S_1(0,q,x)=\frac{6h(q)}{\pi^2}x+O\left((e^4\log{x})^{\omega(q)}\right)+O\left(x^{1/2}\sum_{n\le x}\frac{|a_n|}{n^{1/2}} \right).
\end{align*}
For the last term, 
\begin{align*}
\sum_{n\le x}\frac{|a_n|}{n^{1/2}}\le \prod_{p|q}\left(1-p^{-1/2}\right)^{-1}\le \prod_{p|q}(e^4 \log{x})=(e^4\log{x})^{\omega(q)},
\end{align*}
so that
\begin{equation}
\label{s1}
S_1(0,q,x)=\frac{6h(q)}{\pi^2}x+O\left(x^{1/2}(e^4\log{x})^{\omega(q)}\right).
\end{equation}
Since
\begin{align*}
M(x,q,r,t)&=\frac{1}{t}\displaystyle\sum_{a=0}^{t-1}e_t(-ar)S_1(a,q,x) \\
&=\frac{1}{t}S_1(0,q,x)+\frac{1}{t}\displaystyle\sum_{a=1}^{t-1}e_t(-ar)S_1(a,q,x),
\end{align*}
we have from~\eqref{sa} and~\eqref{s1}
$$M(x,q,r,t)=\frac{6h(q)}{\pi^2t}x+O\left(x^{1/2}(e^4\log{x})^{\omega(q)}\right)+O\left( x (\log{x})^{-\alpha_t}\log \log{q}\right).$$
If $q$ is not squarefree, repeating the above argument with $q$ replaced by its squarefree part gives the general case since the error term is increasing with $q$.
\end{proof}
For complex $s$ we write  $s=\sigma+it$ with both $\sigma$ and $t$ real.
\begin{lemma}
\label{tau and zeta}
Let $m$ be odd and $\chi$ a multiplicative character$\pmod m$. Let 
$$L(s,\chi,\tau)=\displaystyle\sum_{n=1}^{\infty}\frac{\chi(\tau(n))}{n^s},$$
then for $\sigma >1$ we have
$$L(s,\chi,\tau)=\zeta(s)^{\chi(2)}F(s,\chi),$$
 with $F(1,\chi)\neq 0$ and
$$F(s,\chi)=\displaystyle\sum_{n=1}^{\infty}\frac{b(\chi,n)}{n^s},$$
for some constants $b(\chi,n)$ satisfying
\begin{align*}
\displaystyle\sum_{n=1}^{\infty}\frac{|b(\chi,n)|(\log{n})^3}{n}=O(1),
\end{align*}
 uniformly over all characters $\chi$. 
\end{lemma}
\begin{proof}
Since both $\chi$ and $\tau$ are multiplicative we have for $\sigma>1$,
\begin{align*}
L(s,\chi,\tau)&=\displaystyle\prod_{p}\left(1+\displaystyle\sum_{n=1}^{\infty}\frac{\chi(\tau(p^n))}{p^{ns}}\right) \\
&=\zeta(s)^{\chi(2)}F(s,\chi),
\end{align*}
with
\begin{align*}
F(s,\chi)=\displaystyle\prod_{p}\left(1+\displaystyle\sum_{n=1}^{\infty}\frac{\chi(n+1)}{p^{ns}}\right)\left(1-\frac{1}{p^{s}}\right)^{\chi(2)}.
\end{align*}
We have
\begin{align*}
F(s,\chi)&=\displaystyle\prod_{p}\left(1-\frac{\chi(2)}{p^{s}}\right) \left(1+\frac{\chi(2)}{p^{s}}+\displaystyle\sum_{n=2}^{\infty}\frac{\chi(n+1)}{p^{ns}}\right) \times \\
& \qquad \qquad\qquad \qquad\qquad  \displaystyle\prod_{p}\left(1-\frac{\chi(2)}{p^{s}}\right)^{-1} \left(1-\frac{1}{p^{s}}\right)^{\chi(2)}
\\ &=F_1(s,\chi)F_2(s,\chi),
\end{align*}
where
\begin{align}
\label{F_1}
F_1(s,\chi)&=\displaystyle\prod_{p}\left(1-\frac{\chi(2)}{p^{s}}\right) \left(1+\frac{\chi(2)}{p^{s}}+\displaystyle\sum_{n=2}^{\infty}\frac{\chi(n+1)}{p^{ns}}\right) \nonumber \\ &=
\displaystyle\prod_{p}\left(1+\displaystyle\sum_{n=2}^{\infty}\frac{\chi(n+1)-\chi(2n)}{p^{ns}}\right),
\end{align}
and
\begin{align*}
F_2(s,\chi)&=\displaystyle\prod_{p}\left(1-\frac{\chi(2)}{p^{s}}\right)^{-1} \left(1-\frac{1}{p^{s}}\right)^{\chi(2)}.
\end{align*}
Considering $F_2(s,\chi)$, we have for $\sigma>1$
\begin{align}
\label{F_2}
\log{F_2(s,\chi)}&=\displaystyle\sum_{p}\displaystyle\sum_{n=1}^{\infty} \frac{1}{n}\left(\frac{\chi(2^n)} {p^{ns}}-\frac{\chi(2)}{p^{ns}}\right) \nonumber \\
&=\displaystyle\sum_{p}\displaystyle\sum_{n=2}^{\infty}\frac{\chi(2^n)-\chi(2)}{n}\frac{1}{p^{ns}}.
\end{align}
In the equations~\eqref{F_1} and~\eqref{F_2}, the product and the series converge absolutely on $\sigma=1$ so that $F(1,\chi)\neq 0$. Also since   $|\chi(j)-\chi(k)|\leq 2$ for all integers $k,j$ we see that the coefficents $b(\chi,n)$ in
$$F(s,\chi)=\displaystyle\sum_{n=1}^{\infty}\frac{b(\chi,n)}{n^s},$$
satisfy
$$|b(\chi,n)|\leq c_n,$$
where the numbers $c_n$ are defined by
$$\displaystyle\prod_{p}\left(1+\displaystyle\sum_{n=2}^{\infty}\frac{2}{p^{ns}}\right)\exp\left(\displaystyle\sum_{p}\displaystyle\sum_{n=2}^{\infty}\frac{2}{n}\frac{1}{p^{ns}}\right)=\displaystyle\sum_{n=1}^{\infty}\frac{c_n}{n^s}.$$
The function defined by the above formula converges uniformly in any halfplane $\sigma\geq \sigma_0 >1/2$, so that
$$\displaystyle\sum_{n\leq X}c_n=O(X^{1/2+\varepsilon})$$
and the last statement of the Lemma follows by partial summation.
\end{proof}
The following is~\cite[Theorem~7.18]{MgVu}.
\begin{lemma}
\label{Selberg}
Suppose for each complex $z$ we have a sequence $(b_z(n))_{n=1}^{\infty}$ such that the sum
$$\displaystyle\sum_{n=1}^{\infty}\frac{|b_z(n)|(\log{n})^{2R+1}}{n},$$ is uniformly bounded for $|z|\leq R$ and for $\sigma\ge 1$ let 
$$F(s,z)=\displaystyle\sum_{n=1}^{\infty}\frac{b_z(m)}{m^s}.$$
Suppose for $\sigma > 1$ we have 
$$\zeta(s)^zF(s,z)=\displaystyle\sum_{n=1}^{\infty}\frac{a_z(n)}{n^s},$$
for some $a_z(n)$ and let $S_z(x)=\displaystyle\sum_{n\leq x}a_z(n).$ Then for $x\ge 2$, uniformly over all $|z|\leq R$,
$$S_z(x)=\frac{F(1,z)}{\Gamma(z)}x(\log{x})^{z-1}+O(x(\log{x})^{\Re(z)-2}).$$

\end{lemma}
Combining Lemma~\ref{tau and zeta} and Lemma~\ref{Selberg} gives
\begin{lemma}
\label{c tau}
For integer $m$ let $\chi$ be a multiplicative character $\pmod m$ and let $$G(\chi)=\frac{1}{\Gamma(\chi(2))}\displaystyle\prod_{p}\left(\displaystyle\sum_{n=0}^{\infty}\frac{\chi(n+1)}{p^{n}}\right)\left(1-\frac{1}{p}\right)^{\chi(2)}.$$
Then uniformly over all characters $\chi$,
\begin{equation}
\label{neq}
\displaystyle\sum_{n\leq x}\chi(\tau(n))=G(\chi)x(\log{x})^{\chi(2)-1}+O\left(x(\log{x})^{\Re(\chi(2))-2}\right).
\end{equation}
\end{lemma}
\begin{lemma}
\label{tau cong}
For any integer $m$,
$$\displaystyle\sum_{\substack{q \in \mathcal{K} \\ \tau(q) \equiv 0 \pmod{m}}}\frac{1}{q} \ll \frac{1}{m^{\log{2}/2}},$$
and if $p$ is prime
$$\displaystyle\sum_{\substack{q \in \mathcal{K} \\ \tau(q) \equiv 0 \pmod{p}}}\frac{1}{q} \ll \frac{1}{2^{p/2}}.$$
\end{lemma}
\begin{proof}
Suppose $\tau(q) \equiv 0 \pmod{m}$ and let $q=p_1^{\alpha_1}\dots p_k^{\alpha_k}$ be the prime factorization of $q$, so that
$$\tau(q)=(\alpha_1+1)\dots(\alpha_k+1)\ge m.$$
By the arithmetic-geometric mean inequality,
$$q \ge 2^{(\alpha_1+1)+ \dots +(\alpha_k+1)-k}\ge 2^{k(\tau(q)^{1/k}-1)}\ge 2^{k(m^{1/k}-1)}\ge 2^{\log{m}}=m^{\log{2}},$$
and since $\mathcal{K}(x)\ll x^{1/2}$ we get
$$
\displaystyle\sum_{\substack{q \in \mathcal{K} \\ \tau(q) \equiv 0 \pmod{m}}}\frac{1}{q} \le 
\displaystyle\sum_{\substack{q \in \mathcal{K} \\ q \ge m^{\log{2}}}}\frac{1}{q}\ll \displaystyle\int_{m^{\log{2}}}^{\infty}\frac{\mathcal{K}(x)}{x^2}dx\ll \frac{1}{m^{\log{2}/2}}. $$
 Suppose $p$ is prime,  if $\tau(n)\equiv 0 \pmod p$ then $n\ge 2^{p-1}$. As before we get
$$\displaystyle\sum_{\substack{q \in \mathcal{K} \\ \tau(q) \equiv 0 \pmod{p}}}\frac{1}{q}\ll \frac{1}{2^{p/2}}.$$

\end{proof}

\section{Proof of Theorem~\ref{thm:t1}}

By Lemma~\ref{decom} we have
\begin{align*}
\displaystyle\sum_{n=1}^{N}e_m(a\tau(n))&=\displaystyle\sum_{\substack{k\in \mathcal{M}_m \\ k \le N}}
\displaystyle\sum_{\substack{q\in \mathcal{Q}_m \\ q\le N/k}}
\displaystyle\sum_{\substack{s\in \mathcal{S}\\ \gcd(s,q)=1 \\ s\le N/qk}} e_m(a\tau(kqs)) \\
&=\displaystyle\sum_{\substack{k\in \mathcal{M}_m \\ k \le N}}
\displaystyle\sum_{\substack{q\in \mathcal{Q}_m \\ q\le N/k}}
\displaystyle\sum_{\substack{s\in \mathcal{S}\\ \gcd(s,q)=1 \\ s\le N/qk}} e_m(a\tau(q)\tau(s)) \\
&=\displaystyle\sum_{\substack{k\in \mathcal{M}_m \\ k \le N}}
\displaystyle\sum_{\substack{q\in \mathcal{Q}_m \\ q\le N/k}}
\displaystyle\sum_{\substack{s\in \mathcal{S}\\ \gcd(s,q)=1 \\ s\le N/qk}} e_m(a\tau(q)2^{\omega(s)}).
\end{align*}
Let $K=N^{1/2}$ and grouping together values of $2^{\omega(s)}$ in the same residue class$\pmod m$ we have,
\begin{align*}
\displaystyle\sum_{n=1}^{N}e_m(a\tau(n))
&=\displaystyle\sum_{\substack{k\in \mathcal{M}_m \\ k \le N}}
\displaystyle\sum_{\substack{q\in \mathcal{Q}_m \\ q\le K/k}}\displaystyle\sum_{r=1}^{t}M(N/qk,q,r,t)e_m(a\tau(q)2^r)
\\ &\quad \quad \quad+\displaystyle\sum_{\substack{k\in \mathcal{M}_m \\ k \le N}}
\displaystyle\sum_{\substack{q\in \mathcal{Q}_m \\ K/k<q\le N/k}}\displaystyle\sum_{r=1}^{t}M(N/qk,q,r,t)e_m(a\tau(q)2^r).
 \\
\end{align*}
By choice of $K$, we have
$N/qk\ge q$ \  when \  $q\le K/k$.
Hence we may apply Lemma~\ref{Mq} to the first sum above,
\begin{align*}
\displaystyle\sum_{n=1}^{N}e_m(a\tau(n))&=\frac{6}{\pi^2}\frac{N}{t}\displaystyle\sum_{\substack{k\in \mathcal{M}_m \\ k \le N}}\displaystyle\sum_{\substack{q\in \mathcal{Q}_m \\ q\le N/k}}\displaystyle\sum_{r=1}^{t}\frac{h(q)}{qk}e_m(a\tau(q)2^r) \\
&+O\left(\displaystyle\sum_{\substack{k\in \mathcal{M}_m \\ k \le N}}\displaystyle\sum_{\substack{q\in \mathcal{Q}_m \\ q\le K/k}}\frac{Nt\log \log {q}}{qk}(\log{(N/qk)})^{-\alpha_t}\right)+O\left(\displaystyle\sum_{\substack{k\in \mathcal{M}_m \\ k \le N}}
\displaystyle\sum_{\substack{q\in \mathcal{Q}_m \\ K/k<q\le N/k}}\frac{tN}{qk} \right) \\
&\quad \quad \quad+O\left(\displaystyle\sum_{\substack{k\in \mathcal{M}_m \\ k \le N}}\displaystyle\sum_{\substack{q\in \mathcal{Q}_m \\ q\le K/k}}(e^4\log{(N/kq)})^{\omega(q)}\left(\frac{N}{kq}\right)^{1/2}\right).
\end{align*}
Considering the first two error terms,
\begin{align}
\label{bound 1}
\displaystyle\sum_{\substack{k\in \mathcal{M}_m \\ k \le N}}
\displaystyle\sum_{\substack{q\in \mathcal{Q}_m \\ K/k<q\le N/k}}\frac{tN}{qk} \ll tN \int_{K}^{N}\frac{\mathcal{K}(x)}{x^2}dx \ll
\frac{tN}{K^{1/2}},
\end{align}
and since the sum
$$\displaystyle\sum_{\substack{k\in \mathcal{M}_m \\ k \le N}}\displaystyle\sum_{\substack{q\in \mathcal{Q}_m \\ q\le K/k}}\frac{\log \log{q}}{qk},$$
is bounded uniformly in $m$ as $K,N \rightarrow \infty$, we get
\begin{align}
\label{bound 2}
\displaystyle\sum_{\substack{k\in \mathcal{M}_m \\ k \le N}}\displaystyle\sum_{\substack{q\in \mathcal{Q}_m \\ q\le K/k}}\frac{Nt\log{q}}{qk}(\log{(N/qk)})^{-\alpha_t}&\ll Nt\log (N/K)^{-\alpha_t} \displaystyle\sum_{\substack{k\in \mathcal{M}_m \\ k \le N}}\displaystyle\sum_{\substack{q\in \mathcal{Q}_m \\ q\le K/k}}\frac{\log \log{q}}{qk} \nonumber \\ &\ll
Nt(\log(N/K))^{-\alpha_t}.
\end{align}
For the last term,
\begin{align*}
&\displaystyle\sum_{\substack{k\in \mathcal{M}_m \\ k \le N}}\displaystyle\sum_{\substack{q\in \mathcal{Q}_m \\ q\le K/k}}(e^{4}\log{(N/kq)})^{\omega(q)}\left(\frac{N}{kq}\right)^{1/2} \le N^{2/3}\displaystyle\sum_{\substack{n \in \mathcal{K} \\ n\le N }}\left(\frac{1}{n}\right)^{2/3}(e^4\log{N})^{\omega(n)},
\end{align*}
and since $$\omega(n)\le (1+o(1))\frac{\log{n}}{\log\log{n}},$$
we get
\begin{align*}
N^{2/3}\displaystyle\sum_{\substack{k\in \mathcal{M}_m \\ k \le N}}\displaystyle\sum_{\substack{q\in \mathcal{Q}_m \\ q\le K/k}}(e^{4}\log{N})^{\omega(q)}\left(\frac{1}{kq}\right)^{2/3} \le N^{5/6+o(1)}\displaystyle\sum_{\substack{n \in \mathcal{K} \\ n\le N }}\left(\frac{1}{n}\right)^{2/3}(e^4(\log{N})^{5/6})^{\omega(n)}.
\end{align*}
We may bound the sum on the right by noting
\begin{align*}
\displaystyle\sum_{\substack{n \in \mathcal{K} \\ n\le N }}\left(\frac{1}{n}\right)^{2/3}(e^4(\log{N})^{5/6})^{\omega(n)} &\le \displaystyle\prod_{p}\left(1+e^4(\log{N})^{5/6}\displaystyle\sum_{k=2}^{\infty}\frac{1}{p^{2k/3}}\right),
\end{align*}
taking logarithms we see that
\begin{align*}
\log\left(\displaystyle\prod_{p}\left(1+e^4(\log{N})^{5/6}\displaystyle\sum_{k=2}^{\infty}\frac{1}{p^{2k/3}}\right)\right)&=
\displaystyle\sum_{p}\log\left(\left(1+\frac{e^4(\log{N})^{5/6}}{p^{4/3}}\frac{p^{2/3}}{p^{2/3}-1}\right)\right) \\
&\le e^4(\log{N})^{5/6}\sum_{p}\frac{1}{p^{4/3}}\frac{p^{2/3}}{p^{2/3}-1} & \\
&\ll (\log N)^{5/6},
\end{align*}
hence we have for some absolute constant $c$
\begin{equation}
\label{bound 3}
\displaystyle\sum_{\substack{k\in \mathcal{M}_m \\ k \le N}}\displaystyle\sum_{\substack{q\in \mathcal{Q}_m \\ q\le K/k}}(e^{4}\log{(N/kq)})^{\omega(q)}\left(\frac{N}{kq}\right)^{1/2}\ll N^{5/6+o(1)}e^{c(\log{N})^{5/6}}.
\end{equation}
Combining~\eqref{bound 1},~\eqref{bound 2} and~\eqref{bound 3} gives
\begin{align*}
\displaystyle\sum_{n=1}^{N}e_m(a\tau(n))&=\frac{6}{\pi^2}\frac{N}{t}\displaystyle\sum_{\substack{k\in \mathcal{M}_m \\ k \le N}}\frac{1}{k}\displaystyle\sum_{\substack{q\in \mathcal{Q}_m \\ q\le N/k}}\frac{h(q)}{q}\displaystyle\sum_{r=1}^{t}e_m(a\tau(q)2^r)+O\left(Nt(\log(N/K))^{-\alpha_t}\right) \\ & \quad +O\left(\frac{tN}{K^{1/2}}\right)+O\left(N^{5/6+o(1)}e^{c(\log{N})^{5/6}}\right).
\end{align*}
Recalling the choice of $K$ we get
\begin{align}
\label{almost finished}
\displaystyle\sum_{n=1}^{N}e_m(a\tau(n))=\frac{6}{\pi^2}\frac{N}{t}\displaystyle\sum_{\substack{k\in \mathcal{M}_m \\ k \le N}}\frac{1}{k}\displaystyle\sum_{\substack{q\in \mathcal{Q}_m \\ q\le N/k}}\frac{h(q)}{q}\displaystyle\sum_{r=1}^{t}e_m(a\tau(q)2^r)+O\left(Nt(\log{N})^{-\alpha_t}\right).
\end{align}
For the main term,
\begin{align*}
\displaystyle\sum_{\substack{k\in \mathcal{M}_m \\ k \le N}}\frac{1}{k}\displaystyle\sum_{\substack{q\in \mathcal{Q}_m \\ q\le N/k}}\frac{h(q)}{q}\displaystyle\sum_{r=1}^{t}e_m(a\tau(q)2^r)&=\displaystyle\sum_{\substack{k\in \mathcal{M}_m \\ k \le N}}\frac{1}{k}\displaystyle\sum_{q\in \mathcal{Q}_m}\frac{h(q)}{q}\displaystyle\sum_{r=1}^{t}e_m(a\tau(q)2^r) \\
&+O\left(t\sum_{\substack{k\in \mathcal{M}_m \\ k \le N}}\frac{1}{k}\left(\frac{k}{N}\right)^{1/2}\right) \\
&=\displaystyle\sum_{k\in \mathcal{M}_m }\frac{1}{k}\displaystyle\sum_{q\in \mathcal{Q}_m}\frac{h(q)}{q}\displaystyle\sum_{r=1}^{t}e_m(a\tau(q)2^r) \\
&+O\left(\frac{t}{N^{1/2}}\right) \\
&=\zeta(m)\displaystyle\sum_{q\in \mathcal{Q}_m}\frac{h(q)}{q}\displaystyle\sum_{r=1}^{t}e_m(a\tau(q)2^r)+O\left(\frac{t}
{N^{1/2}}\right).
\end{align*}
Hence we have
$$\displaystyle\sum_{n=1}^{N}e_m(a\tau(n))=\frac{\zeta{(m)}}{t}\frac{6}{\pi^2}\left(\displaystyle\sum_{r \pmod m}H(r,m)S_m(ar)\right)N+O\left(tN(\log{N})^{-\alpha_t}\right).$$

\section{Proof of Theorem~\ref{thm:t2}}
Let 
$$C(p,r,N)= \# \{ n\le N : \tau(n) \equiv r \pmod p \},$$
 so that  
\begin{equation}
\label{con exp}
\displaystyle\sum_{n=1}^{N}e_p{(a\tau(n))}=\displaystyle\sum_{r=0}^{p-1}C(p,r,N)e_m(an).
\end{equation}
Suppose $(r,p)=1$, using orthogonality of characters and Lemma~\ref{c tau},
\begin{align*}
C(p,r,N)&=\frac{1}{p-1}\sum_{n=1}^{N}\displaystyle\sum_{\chi (\text{mod} \ p)}\overline\chi(r)\displaystyle\chi(\tau(n)) \\
&=\frac{1}{p-1}\sum_{\chi(2)=1}\overline\chi(r)G(\chi)N
+\frac{1}{p-1}\sum_{\chi(2)\neq  1}\overline\chi(r)G(\chi)N(\log{N})^{\chi(2)-1} \\  & \quad +O\left(N (\log{N})^{-(\alpha_p+1)}\right),
\end{align*}
 and for $r=0$, we have by another applications of  Lemma~\ref{c tau}
$$C(p,0,N)=\displaystyle\sum_{n=1}^{N}(1-\chi_0(\tau(n)))=C_pN+O\left(N (\log{N})^{-2}\right),$$
for some constant $C_p$.
Hence from~\ref{con exp}
\begin{align*}
\displaystyle\sum_{n=1}^{N}e_p{(a\tau(n))}&=C_pN+O\left(pN(\log{N})^{-(\alpha_t+1)}\right) \\ &+
\frac{N}{p-1}\sum_{r=1}^{p-1}\left(\sum_{\chi(2)=1}\overline\chi(r)G(\chi)e_p(ar)
+\sum_{\chi(2)\neq 1}\overline\chi(r)e_p(ar)G(\chi)(\log{N})^{\chi(2)-1}\right) \\
&=A_pN+\frac{N}{p-1}\sum_{\chi(2) \neq 1}G(\chi)(\log{N})^{\chi(2)-1}\sum_{r=1}^{p-1}\overline\chi(r)e_p(ar) \\ & \quad +O\left(pN(\log{N})^{-(\alpha_t+1)}\right),
\end{align*}
for some constant $A_p$. If $\chi(2)\neq 1$  then we have $$\left|\displaystyle\sum_{r=1}^{p-1}\overline\chi(r)e_p(ar)\right|=p^{1/2},$$ so that
\begin{align}
\label{AAAAAAAAAA}
\displaystyle\sum_{n=1}^{N}e_p{(a\tau(n))}&=A_p N +O\left(p^{1/2}N(\log{N})^{-\alpha_t}\right)+O\left(pN(\log{N})^{-(\alpha_t+1)}\right) \nonumber \\
&=A_p N +O\left(pN(\log{N})^{-(\alpha_t+1)}\right),
\end{align}
since if $$p^{1/2}N(\log{N})^{-\alpha_t}\le pN(\log{N})^{-(\alpha_t+1)} \quad \text{then} \quad N\le  pN(\log{N})^{-(\alpha_t+1)}.$$
Finally, comparing~\eqref{AAAAAAAAAA} with the leading term in the asymptotic formula from Theorem~\ref{thm:t1}, we see that 
$$A_p=\frac{\zeta{(p)}}{t}\frac{6}{\pi^2}\left(\displaystyle\sum_{r=0}^{p-1}H(r,p)S_{p}(ar)\right).$$ 

\section{Proof of Theorem~\ref{ub m}}

Considering the main term in Theorem~\ref{thm:t1}
\begin{align*}
\left|\displaystyle\sum_{r=0}^{m-1}H(r,m)S_{m}(ar)\right|&\le\sum_{d | m}\left|\sum_{\substack{r=0 \\ \gcd(r,m)=d}}^{m-1}H(r,m)S_{m}(ar)\right| \\
&\le \sum_{d | m}\left(\sum_{\substack{r=0 \\ \gcd(r,m)=d}}^{m-1}H(r,m)\right)\max_{\gcd(\lambda,m)=d}|S_{m}(\lambda)|.
\end{align*}
Writing $c=\log{2}/2$, by Lemma~\ref{tau cong}
\begin{align*}
\sum_{\substack{r=0 \\ \gcd(r,m)=d}}^{m-1}H(r,m)&=\sum_{\substack{r=0 \\ \gcd(r,m)=d}}^{m-1}
\displaystyle\sum_{\substack{q \in \mathcal{Q}_m \\ \tau(q) \equiv r \pmod m}}\frac{h(q)}{q} \\
&\le \displaystyle\sum_{\substack{ q \in \mathcal{K} \\ \tau(q) \equiv 0 \pmod{d}}}\frac{1}{q} \ll \dfrac{1}{d^{c}},
\end{align*}
so that 
\begin{equation}
\label{coefficient bound}
\displaystyle\sum_{r=0}^{m-1}H(r,m)S_{m}(ar)\ll \sum_{d | m}\frac{1}{d^c}\max_{\gcd(\lambda,m)=d}|S_{m}(\lambda)|.
\end{equation}
Suppose $\gcd(\lambda,m)=d$, so that $\lambda=d\lambda'$ and $m=dm'$ for some $\lambda'$ and $m'$ with $\gcd(\lambda',m')=1$. Let $t_d$ denote the order of $2 \pmod {m'}$. Then we have
$$S_{m}(\lambda)=\sum_{n=1}^{t}e_m(\lambda 2^n)=\frac{t}{t_d}\sum_{n=1}^{t_d}e_{m'}(\lambda' 2^n).$$
By the main result of \cite{Bou}, if $t_d \ge (m/d)^{\varepsilon}$ then for some $\delta>0$,
\begin{equation}
\label{coefficient bound 1}
S_{m}(\lambda)\ll \left(\frac{d}{m}\right)^{\delta}t.
\end{equation}
Suppose  $t\ge m^{\varepsilon}$, then since
$$t_d \ge \dfrac{t}{d}\ge \frac{m^{\varepsilon/2}}{d}m^{\varepsilon/2},$$
if $d \le m^{\varepsilon/2}$ then we have
$t_d \ge (m/d)^{\varepsilon/2}$. Hence by~\eqref{coefficient bound 1}
$$S_{m}(\lambda)\ll \left(\frac{d}{m}\right)^{\delta}t\ll \frac{t}{m^{\delta_0}}.$$
Hence by~\eqref{coefficient bound}, for some $\delta_1>0$
\begin{align}
\label{coefficient final bound}
\displaystyle\sum_{r=0}^{m-1}H(r,m)S_{m}(ar)&\ll \sum_{\substack{d | m \\ d\le m^{\varepsilon/2}}}\frac{1}{d^c}\max_{\gcd(\lambda,m)=d}|S_{m}(\lambda)|
+\sum_{\substack{d | m \\ d \ge m^{\varepsilon/2}}}\frac{1}{d^c}\max_{\gcd(\lambda,m)=d}\left|S_{m}(\lambda)\right| \nonumber \\
& \ll \sum_{\substack{d | m \\ d\le m^{\varepsilon/2}}}\frac{t}{m^{\delta_1}}
+\sum_{\substack{d | m \\ d \ge m^{\varepsilon/2}}}\frac{1}{m^{\delta_1}}=\frac{\tau(m)}{m^{\delta_1}}t,
\end{align}
and the result follows combining~\eqref{coefficient final bound} with Theorem~\ref{thm:t1}.  

\section{Proof of Theorem~\ref{ub p}}

By Lemma~\ref{tau cong}
\begin{align*}
\displaystyle\sum_{r=0}^{p-1}H(r,p)S_{p}(ar,t)&=\displaystyle\sum_{r=1}^{p-1}H(r,p)S_{p}(ar)+tH(0,p) \\
&\ll \frac{t}{2^{p/2}}+\left(\displaystyle\sum_{r=1}^{p-1}H(r,p) \right)\max_{\gcd(\lambda,p)=1}\left|S_{p}(\lambda)\right| \\
& \ll \frac{t}{2^{p/2}}+\max_{\gcd(\lambda,p)=1}\left|S_{p}(\lambda)\right|. 
\end{align*}
In~\cite{Kerr} it is shown the following bound is a consequence of~\cite{Kor} and~\cite{Sk},
$$\max_{\gcd(\lambda,p)=1}|S_{p}(\lambda)| \ll \begin{cases}  p^{1/8}t^{22/36}(\log{p})^{7/6}, \quad  \ \  \  \ t \le p^{1/2}, \\  
 p^{1/4}t^{13/36}(\log{p})^{7/6}, \quad \  \ p^{1/2}<t \le p^{3/5}(\log{p})^{-6/5}, \\ 
 p^{1/6}t^{1/2}(\log{p})^{4/3}, \quad \ \ \ \ \  p^{3/5}<t \le p^{2/3}(\log{p})^{-2/3}, \\
p^{1/2}\log{p}, \quad \quad \quad  \ \ \ \ \  \ \ \  \ \ t>p^{2/3}(\log{p})^{-2/3}, \end{cases}$$
and the result follows by Theorem~\ref{thm:t2}.

\end{document}